\documentclass[12pt]{amsart}
\usepackage{amsmath,amsthm,amsfonts,amssymb,eucal}

\renewcommand {\a}{ \alpha }

\newcommand{\y}{\eta}

\newcommand{\G}{\Gamma}

\newcommand{\vark}{\varkappa}
\newcommand{\varf}{\varphi}
\renewcommand{\d}{\delta}

\newcommand{\D}{\Delta}
\newcommand{\s}{\sigma}
\newcommand{\Sg}{\Sigma}
\renewcommand{\l}{\lambda}

\renewcommand{\t}{\tau}

\newcommand{\Om}{\Omega}

\newcommand{\R}{ \mathbb R}

\newcommand{\Z}{ \mathbb Z}

\newcommand{\CL}{\mathcal L}

\newcommand{\CD}{\mathcal D}
\newcommand{\CE}{\mathcal E}

\newcommand{\CH}{\mathcal H}

\newcommand{\CP}{\mathcal P}

\newcommand{\CS}{\mathcal S}

\newcommand{\CV}{\mathcal V}

\newcommand {\GS}{\mathfrak S}

\newcommand {\ba}{\mathbf a}

\newcommand {\bd}{\mathbf d}

\newcommand {\bb}{\mathbf b}

\newcommand {\BA}{\mathbf A}
\newcommand {\BB}{\mathbf B}

\newcommand {\BI}{\mathbf I}

\newcommand {\BP}{\mathbf P}

\newcommand{\wh}{\widehat}

\DeclareMathOperator{\dom}{Dom} 
\DeclareMathOperator{\re}{Re}

\DeclareMathOperator{\comp}{comp}
\DeclareMathOperator{\fin}{fin}

\DeclareMathOperator{\pl}{pl}
\DeclareMathOperator{\Dom}{Dom}
\newtheorem{thm}{Theorem}[section]

\newtheorem{lem}[thm]{Lemma}

\theoremstyle{definition}

\theoremstyle{remark}
\newtheorem{rem}[thm]{Remark}

\numberwithin{equation}{section}

\newcommand{\thmref}[1]{Theorem~\ref{#1}}

\newcommand{\bsymb}{\boldsymbol}

\newcommand{\qf} {quadratic form }

\newcommand{\vs}{\vskip0.2cm}

\begin{document}
\title[Spectral estimates on spaces of small local dimension]
{On the spectral estimates for the Schr\"odinger type operators:\\
the case of small local dimension}
\author[G. Rozenblum]{G. Rozenblum}
\address{Department of Mathematics \\
                        Chalmers University of Technology
                        and  The University of Gothenburg \\
                         S-412 96, Gothenburg, Sweden}
\email{grigori@chalmers.se}
\author[M. Solomyak]{M. Solomyak}
\address{Department of Mathematics
\\ Weizmann Institute\\ Rehovot\\ Israel}
\email{solom@wisdom.weizmann.ac.il}

\dedicatory{In  memory of M.Sh. Birman, a Scientist and a Person}

\begin{abstract}
The behavior of the discrete spectrum of the Schr\"odinger operator $-\D - V$, in
quite a general setting, to a large extent is determined by the behavior of the corresponding
heat kernel $P(t;x,y)$ as $t\to 0$ and $t\to\infty$. If this behavior is powerlike, i.e.,
\[\|P(t;\cdot,\cdot)\|_{L^\infty}=O(t^{-\d/2}),\ t\to 0;
\|P(t;\cdot,\cdot)\|_{L^\infty}=O(t^{-D/2}),\ t\to\infty,\]
then it is natural to call the exponents $\d,D$ `{\it the local dimension}'
and `{\it the dimension at infinity}', respectively. The character of spectral estimates
depends on the relation between these dimensions. We analyze the case where
$\d<D$, insufficiently studied before. Our applications concern both combinatorial
and  metric graphs.
\end{abstract}
\maketitle

\section{Introduction}\label{intro}
One of the most influential papers by M.Sh.Birman has been
\cite{Bir61} (1961). The approach developed there, under
the name `The Birman-Schwinger Principle', has been the source of
inspiration and one of the main tools in the
  spectral analysis for Schr\"odinger type operators.

In  \cite{RS97} this tool was applied to
eigenvalue estimates for such operators in a very general
setting, and it turned out that these estimates depend essentially
on  two numerical characteristics of the operator, $\d$ and $D$, that can be called
the local dimension
 and the dimension at infinity. For the standard Schr\"odinger operator on $\R^d$,
these characteristics coincide with the dimension; in general, $\d\ne D$.

In the survey paper \cite{RS08} various relations between dimensions were discussed and the main attention was
given there to the effects appearing when $\d\ge D$. The simplest example with $\d<D$
is given by the lattice $\Z^d$ (i.e., by the discrete
Schr\"odinger operator); here $D=d$ and $\d=0$.
In this case some peculiarities in the spectral
distribution were discovered in \cite{RS09}. In particular, the large coupling constant eigenvalue estimates, that
are order sharp in $\R^d$, are not sharp in $\Z^d$ any more.
Deeper deliberations on the effects found in \cite{RS09} have lead the
authors to understanding that these peculiarities are common to
all situations when $\d<D$. In the present paper we consider the
cases of rather general combinatorial graphs, where $\d=0$, and quantum (metric) graphs, where  $\d>1$ is arbitrarily close to $1$; we restrict ourselves to the situation where $D>2$.
We find a range of spectral estimates for Schr\"odinger type
operators on such graphs. For quantum graphs, it
turns out that  such estimates are determined by the
corresponding estimates for the associated combinatorial graph,
which is rather unexpected since the quantum graph contains much
more `flesh'. This phenomenon is supported by Theorem 4.1, where the relations between dimensions of these graphs are established.
We find also criteria for the Birman-Schwinger operator to belong to various Schatten,
or `weak' Schatten
classes, and conditions for the validity of a Weyl type eigenvalue asymptotics.
The main results in this direction are Theorems 5.3 and 5.5.

The first-named author (G.R.) expresses his gratitude to The Weizmann Institute of Science for hospitality and support. The authors are grateful to T.A.~Suslina and the Referees for a careful reading and useful advices.

\section{Operators on general measure spaces}\label{set}
\subsection{Dimensions of a semigroup}Let $(X,\s)$ be a measure space with sigma-finite measure. We
denote $L^q(X)=L^q(X,\s)$ and
$\|\cdot\|_q=\|\cdot\|_{L^q(X)}$. Often we drop the symbol $X$ in our notation.
Let $\BA$ be a non-negative self-adjoint operator
in $L^2(X)$ and $\BP(t)=\exp(-\BA t)$ be the corresponding semigroup. We assume that for
any $t>0$ the operator $\BP(t)$ is positivity preserving and is bounded as acting
from $L^1$ to $L^\infty$. It is well known that under these assumptions
$\BP(t)$ is an integral operator whose kernel $P(t;x,y)$
(heat kernel) is well-defined for $t>0$ as a function
in $L^\infty( X\times X)$ (see \cite{RS97} for details). We denote
 \begin{equation*}  M_\BA(t)=\|P(t;\cdot,\cdot)\|_{L^\infty(X\times X)}.\end{equation*}
The above described class of operators $\BA$ (we use the notation $\CP$ for it) includes
the Laplacian, both in its continuous and discrete versions, and also many other important operators;
see e.g. \cite{RS97, MV}. For simplicity, we use below the term `Laplacian' for  any operator
$\BA\in\CP$.

The function $M_\BA(t)$ is non-increasing. Its main characteristics are
the behavior as $t\to 0$ and as $t\to\infty$. In this paper we always
assume that there are two non-negative exponents $\d,D$, such that
\begin{equation}\label{exp0}
    M_\BA(t)=O(t^{-\d/2}),\ t\to0;\qquad M_\BA(t)=O(t^{-D/2}),\ t\to\infty,
\end{equation}
and moreover,
\begin{equation}\label{exp1}
    D\ge\d,\qquad D>2.
\end{equation}
Denote by $\ba$ the quadratic form of $\BA$.
It follows from \eqref{exp0}, \eqref{exp1} that $M_\BA(t)\le Ct^{-D/2}$,
with some $C$, for all
$t\in(0,\infty)$, and hence, by the Varopoulos theory, see
\cite{VCS},
the last inequality is equivalent to the `embedding theorem'
\begin{equation}\label{embed}\|u\|^2_{L^p}\le C\ba [u],\ p=p(D)=2D/(D-2),\qquad \forall u\in\dom(\ba).\end{equation}
We denote by $H_\ba$ the completion of the domain of $\ba$ with
respect to the `$\ba$-norm'
$\|u\|_\ba=\sqrt{\ba[u]}$. By \eqref{embed}, $H_\ba$ can be realized as a space of functions on $X$
(embedded in $L^p$).

\subsection{General eigenvalue estimates}
Let $V\ge0$ be a measurable function on $X$. Under some additional assumptions, the operator
$\BA-V$ defined via its quadratic form $\ba[u]-\int_X V|u|^2d\s$  is self-adjoint, with negative spectrum  consisting of a finite number of eigenvalues of finite multiplicities. Below $N_-(\BA-V)$ stands for the total multiplicity of the negative spectrum of $\BA-V$. It is proved in \cite{RS97} that the number $N_-(\BA-V)$ can be conveniently estimated in terms of the function $M_\BA(t)$. This estimate is an abstract version of Lieb's approach to the proof of the
Rozenblum -- Lieb -- Cwikel (RLC) estimate.

The estimates for $N_-(\BA-V)$ depend on which one of the exponents $\d,D$ in \eqref{exp0} is bigger than the other one.
In particular, the following version of RLC estimate is valid (see, e.g., \cite{RS08}, Remark 1 in section 3.2, see also \cite{RS97} and especially section 3.1 there -- with $d$
replaced by $D$ -- for the details of the proof). We write it down for the operator with a large parameter $\a>0$ (the coupling constant)
incorporated.

\begin{thm}\label{rlc}
Suppose that $D\ge\d,\ D>2$. Then for any $V\in L^{D/2}(X)$ and any $\a>0$
the operator $\BA-\a V$ is well-defined, its negative spectrum is finite, and the following estimate is
satisfied:
\begin{equation}\label{estrlc}
    N_-(\BA-\a V)\le C\a^{D/2}\int_X V^{D/2}d\s,\qquad C=C(X,D).
\end{equation}
\end{thm}
It is convenient to formulate the estimates of this type in terms of the corresponding
Birman -- Schwinger operator $\BB_V$. Recall (see, e.g., \cite{Bir61, BS}) that $\BB_V$ is the
operator  in $H_\ba$ generated by the quadratic form
\begin{equation}\label{b}
    \bb_V[u]=\int_X V|u|^2d\s.
\end{equation}
 Due to the  inequality \eqref{embed}, for $V\in L^{D/2}$ this operator is well defined. The Rayleigh quotient for  $\BB_V$ is
\[ \bb_V[u]/\ba[u],\qquad u\in H_\ba;\]
its eigenvalue counting function is denoted by $n(s,\BB_V).$

The next result is an equivalent reformulation of \thmref{rlc}.
\begin{thm}\label{rlc1}
Under the assumptions of  \thmref{rlc}, $\BB_V\in\Sg_{D/2}$, and
\begin{equation}\label{rlsequiv}
    \|\BB_V\|_{\Sg_{D/2}}\le C\|V\|_{L^{D/2}}.
\end{equation}
\end{thm}
We recall that $\Sg_p$ stands for the class of all compact operators with the powerlike estimate for
the $s$-numbers, $s_n(T)=O(n^{-1/p})$, see \cite{BSbook}, \S 11.6. Similar classes with
$o$ in place of $O$ are denoted by $\Sg_p^{(0)}$, they are closed in $\Sg_p$.  Below we  also use the standard Neumann -- Schatten
classes $\GS_p, 0<p\le\infty$.

Suppose now that $\d<D$. Then \eqref{exp0} implies $M_\BA(t)=O(t^{-q}),\ t\to0,\infty$, with any $q\in [\d/2,D/2]$.
 It follows that an estimate similar to \eqref{rlsequiv} but with any such exponent $q$ instead of $D/2$
(and with a constant depending on $q$) is also valid:
\begin{equation}\label{estrlc1}
     \|\BB_V\|_{\Sg_q}\le C\|V\|_{L^q},\qquad \d/2\le q\le D/2,\ q>1.
\end{equation}

In the case of the Euclidean Laplacian on $\R^d,\ d\ge3$ (here $D=\d=d$) the estimate \eqref{rlsequiv}
is known to be sharp, in the sense that for $V\not\equiv 0$ the operator $\BB_V$ cannot belong to any class, smaller than
$\Sg_{d/2}$. It was shown in \cite{RS09} that for the lattice Laplacian the situation is different:
$V\in\ell^{d/2}(\Z^d),\ d\ge3$, yields $\BB_V\in\Sg_{d/2}^{(0)}$ (or, in other terms,
$N_-(\BA-\a V)=o(\a^{d/2})$). Our next result shows that a similar fact takes place in the general case, provided that $D>\d$.
\begin{thm}\label{o} Suppose that in the assumptions of \thmref{rlc} we have $D>\d$. Then
$\BB_V\in\Sg_{D/2}^{(0)}$.
\end{thm}
\begin{proof}
Fix a number $q\in(\max(\d/2,1),D/2)$. Functions $V\in L^{D/2}\cap L^{\d/2}$ belong to $L^q$ and form a dense subset in $L^{D/2}$. For such functions $V$, the estimate \eqref{estrlc1} implies $\BB_V\in\Sg_{q}\subset\Sg_{D/2}^{(0)}$.
By continuity (see \cite{BSbook}, Theorem 11.6.7), this inclusion carries over to any $V\in L^{D/2}$.
\end{proof}

The following result shows that the estimate \eqref{estrlc1} with any $q<D/2$ is not sharp, in
the sense that the class of admissible potentials can be considerably widened, and the operators
$\BB_V$ for $V\in L^q$  belong actually  to the class $\GS_q$ which is smaller than $\Sigma_q$.
Below $L^q_w$ stands
for the weak $L^q$-space (see, e.g., \cite{BergLof}).
\begin{thm}\label{weakD} Suppose that in the assumptions of
\thmref{rlc} we have $D>\d$. Then
for any $q\in(\max(\d/2,1),D/2)$
\[ V\in L^{q}_w\Longrightarrow \BB_V\in \Sg_{q};\qquad V\in L^{q}
\Longrightarrow \BB_V\in \GS_{q},\]
with the estimates
\begin{equation}\label{XX}
    \|\BB_V\|_{\Sg_{q}}\le C \|V\|_{L^{q}_w};\qquad
\|\BB_V\|_{\GS_{q}}\le C \|V\|_{L^{q}}.
\end{equation}
Equivalently,
\begin{equation}\label{estweak}
    N_-(\BA-\a V)\le C\a^{q}\|V\|^{q}_{L^{q}_w}
\end{equation}
and  $N_-(\BA-\a V)= o(\a^{q}) , \a\to\infty,$ for $ V\in L^{q}$.
\end{thm}
The result follows from \eqref{estrlc1} by the real interpolation (see \cite{BergLof}).

\section{Combinatorial graphs}\label{combinatorial} In the rest of the paper
we consider operators on  graphs. In this section we treat
combinatorial graphs (notation $G$), and in the next two sections we discuss
metric graphs (notation $\G$).  In both cases we \emph{always} assume that the
graph is connected, has an infinite number of vertices, and has no loops, vertices with degree one, or multiple edges (see, e.g., \cite{H} for the main notions of the graph theory).
 We denote the set of edges by $\CE$ and the set of vertices by
$\CV$. The notation $v\sim v'$ means that the vertices $v,v'$ are connected by an edge
that we sometimes denote by $(v,v')$. We \emph{always} suppose that  degrees of all vertices are finite:
   $$ \deg(v)=\#\{v'\in\CV,v'\sim v\}<\infty,\ \forall v\in\CV.$$
With each edge $e\in\CE$ we associate
a  weight
$g_e>0$. We need such `weighted graphs' when dealing with  metric graphs in
 Sect. 4, 5.

On the set $\CV$ we consider the counting measure $\sigma:  \s(v)=1$ for any $v\in\CV$. The basic Hilbert space in this section is
$\ell^2(\CV)=L^2(\CV,\s)$; we write also $\ell^q=L^q(\CV,\s)$, $\ell^q_w=L^q_w(\CV,\s)$. The quadratic form
\begin{equation}\label{form}
    \ba_G[f]=\sum_{e\in\CE; e=(v,v')}g_e|f(v)-f(v')|^2,
\end{equation}
with the  domain $f\in\ell^2(\CV), \ba_G[f]<\infty$, defines in $\ell^2(\CV)$ a nonnegative
self-adjoint operator, $\BA=\BA_G=-\D_G$. In particular, if the weights $g_e$ and the degrees $\deg(v)$ are uniformly bounded, the operator $\BA$ is bounded. Due to the inclusions
$\ell^1(\CV)\subset \ell^2(\CV)\subset \ell^\infty(\CV)$, the
operators $\exp(-\BA t)$ are bounded
as acting from $\ell^1$ to $\ell^\infty$, so that
    $M_\BA(t)\le C.$
This means that $\d=0$. Our main assumption (cf. \eqref{exp0}, \eqref{exp1}) is that
\begin{equation}\label{D}
    M_{\BA_G}(t)=O(t^{-D/2}),\ t\to\infty,\qquad {\text{with some}}\ D>2.
\end{equation}
The corresponding inequality \eqref{embed}:
\begin{equation}\label{Sobolev}
    \|f\|_{\ell^p}^2\le C\ba_G[f],\qquad p=p(D)=2D(D-2)^{-1}
\end{equation}
is certainly satisfied for all $f$ with finite support. Hence, the closure $\CH(G)$ of
the set of all such functions in the metric
$\ba_G[f]$ is embedded in $\ell^p$. In notation of section \ref{set}, $\CH(G)$ plays the role of the space $H_\ba$. In what follows, $\CH_{\fin}(G)$ stands for the set of all finitely supported functions
in $\CH(G)$,  considered as a linear subspace
in $\CH(G)$.

There are many geometric and analytic criteria for the
relation \eqref{D} to hold. Without going into details, we
refer to  \cite{Coulh2001}, \cite{GrigHu08}, \cite{Saloff09}, where
such criteria are presented. An example of
graph satisfying  \eqref{D} is the integer lattice $\Z^d$; for each edge $e$ we take $g_e=1$. Here $D=d$,
which can be
checked by a direct computation of the heat kernel. This case was the object
of our study in the paper \cite{RS09}. Here we extend some of its results
to general graphs.

The following result, in its main part, is just a special case of Theorems 2.2 -- 2.4.\begin{footnote}{From now on, we formulate the results in the terms of the Birman-Schwinger operator only.}\end{footnote} The only essential novelty is that the condition $q>1$ that implicitly appears in  Theorem  2.4 is no longer necessary.
Note that now $\bb_V$ takes the form
\begin{equation}\label{bdiscr}
    \bb_V[f]=\sum_{v\in\CV}V(v)|f(v)|^2.
\end{equation}

\begin{thm}\label{thm1} Let \eqref{D} be satisfied.

$1^\circ$ Suppose $V\in\ell^{D/2}(\CV)$. Then
\begin{equation}\label{RLCdiscr}\|\BB_V\|_{\Sg_{D/2}}\le C\|V\|_{\ell^{D/2}},
\end{equation}
and in addition, $\BB_V\in\Sg_{D/2}^{(0)}$.

$2^\circ$ If $V\in \ell^{q}_w(G)$ $($or $V\in \ell^{q}(G)$ $)$ for some
$q\in(0,D/2)$ then $\BB_V\in\Sg_{q}$ $($resp.,
 $\BB_V\in\GS_{q}$ $)$, and the estimates \eqref{XX} hold true.
\end{thm}
The proof  is the same as for  $G=\Z^d$, see \cite{RS09}, and we skip it.

In contrast to the general situation of section \ref{set}, for graphs it is also possible to obtain a lower  bound for $n(s,\BB_V)$ in terms of the distribution function for $V$, i.e.,
\[ \nu(\t, V)=\# E(\t,V);\ E(\t,V)=\{v\in\CV: V(v)>\t\}, \qquad \t>0.\]
This estimate does not require any
 assumptions about $V$. We need, however, two additional
assumptions about the graph $G$: the weights $g_e$ should be
uniformly bounded:
\begin{equation}\label{g_e}
    g_e\le g_0,
\end{equation}
and the degrees of the vertices should be
 uniformly
bounded:
\begin{equation}\label{deg}
    \deg(v)=\#\{v'\in\CV,v'\sim v\}\le \bd.
\end{equation}
\begin{thm}\label{snizu}
Let    \eqref{g_e}, \eqref{deg} be satisfied for a graph $G$. Then for any $V\ge0$
\begin{equation}\label{ocenkSnizu}
    n(s,\BB_V)\ge (\bd+1)^{-1}\nu(g_0(\bd+1) s, V),
\end{equation}
and therefore, for any $q>0$,
\begin{equation}\label{snizu w klassah}
    \|\BB_V\|_{\Sigma_{q}}\ge c \|V\|_{\ell^{q}_w},\; \|\BB_V\|_{\GS_{q}}\ge c
    \|V\|_{\ell^{q}};\; c=c(q,\bd,g_0)>0.
\end{equation}
\end{thm}
\begin{proof} It is well  known that the set $\CV$ can be broken
 into the union of no more than $\bd+1$ disjoint subsets $\CV_j$, so that no pair of
 vertices in the same subset is connected  by an edge. Therefore,
 for any fixed $\t>0$ the set $E(\t,V)=\{v\in\CV: V(v)>\t\}$ splits into the union of
no more than $\bd+1$ disjoint subsets $\Om_j=E(\t,V)\cap \CV_j$.
 For at least one of
 them, say $\Om_1$, we have $\#\Om_1\ge (\bd+1)^{-1}\nu(\t,V)$.
 Now, consider the subspace $\CL\subset \CH(G)$ generated by the
 functions $f_{v'}(v)=\d_{v,v'},\ v'\in\Om_1$.
These functions are mutually orthogonal
both in the metric \eqref{form} and with respect to the quadratic
form $\bb_V$ in \eqref{bdiscr}. So, for any $f(v)=\sum_{v'\in\Om_1}c_{v'}\d_{v,v'} \in\CL$ we have
\begin{equation}\label{snizu5}
\ba_G[f]=\sum_{v\in{\Om_1}} |c_{v}|^2\sum_{e\ni v}g_e\le
g_0(\bd+1) \sum_{v\in \Om_1}|c_v|^2,
\end{equation}
 while $\bb_V[f]=
\sum_{v\in\Om_1}|c_v|^2V(v)\ge \t \sum_{v\in\Om_1} |c_v|^2$. So we have constructed the subspace
$\CL,\ $  $\dim \CL\ge (\bd+1)^{-1}\nu(\t,V)$, on which
$$\bb_V[f]\ge \t g_0^{-1} (\bd+1)^{-1} \ba_G[f].$$
This implies \eqref{ocenkSnizu} by the variational principle. The estimates \eqref{snizu w klassah}
follow from \eqref{ocenkSnizu} in a standard way.
\end{proof}

\section{Metric graphs: upper estimates}\label{prel}
\subsection{The Laplacian and the decomposition of the space}Each edge $e$ of a metric graph $\G$ is
considered as a line segment of the length $l_e>0$. With $\G$ we associate
the combinatorial graph
$G=G(\G)$, with the same set of vertices $\CV$, the same set of
edges $\CE$, and the same connection relations. To any edge $e$ of $G(\G)$
we assign the weight $g_e=l_e^{-1}$. If $v\in \CV$, then $\CS(v)$ stands for its {\it star}:
$\CS(v)=\cup_{e\ni v}e$. If $e=(v,v')$, then we define $\CS(e)=\CS(v)\cup\CS(v')$.

The Lebesgue measure on the edges induces a measure on $\G$, and our
basic Hilbert space is $L^2(\G)$ with respect to this measure.

On  the space $H^1(\G)$ of continuous functions $\varf$ on $\G$, such that $\varf\in H^1(e)$
on each edge and $\int_\G(|\varf'|^2+|\varf|^2)dy<\infty$,
we consider the quadratic form
 \begin{equation}\label{QformGamma}
\ba_\G[\varf]:=\int_\G|\varf'(y)|^2 dy.
\end{equation}
The Laplacian $\BA_\G$ in $L^2(\G)$ is determined  by this quadratic form. It acts as $-\frac{d^2}{dy^2}$ on each edge. Its domain, $\Dom(\BA_\G)$, consists of all functions $\varf$ belonging to $H^2$ on each edge, continuous at all vertices, satisfying Kirchhoff conditions, and
such that
\[\sum_{e\in\CE}\|\varf\|^2_{H^2(e)}<\infty.\]
First we consider the relations between the exponents $\d,D$ for the semigroups
generated by the operators $\BA_\G$ in $L^2(\G)$ and $\BA_{G(\G)}$ in $\ell^2(G(\G))$.

To this end, let us consider two pre-Hilbert spaces that are linear subspaces in
the space $H^1_{\comp}(\G)$ of all compactly supported functions from $H^1(\G)$. One of them,
$H^1_{\comp,\pl}(\G)$, is formed by  functions linear on each edge;
the subscript $\pl$ stands for `piecewise-linear'. Any function
$\varf\in H^1_{\comp,\pl}(\G)$ is determined by its values $\varf(v)$ at the
vertices. Given a  sequence $f=\{f(v)\},\ v\in\CV$ with finite support, we denote by
$Jf$ the (unique) function in $H^1_{\comp,\pl}(\G)$, such that
$(Jf)(v)=f(v),\ \forall v\in\CV$. The mapping
$J$ defines an isometry between the pre-Hilbert spaces $H^1_{\comp,\pl}(\G)$
equipped with the metric $\ba_\G$, and $\CH_{\fin}(G)$, equipped with the
metric $\ba_G$ defined in \eqref{form}. By means of  this isometry
we identify   these pre-Hilbert spaces.

Another subspace is $H^1_{\comp,\CD}$  consisting of all functions
$\varf\in H^1_{\comp}(\G)$, such that $\varf(v)=0$ for all $v\in \CV$.
It is clear that
\begin{equation}\label{1}
   H^1_{\comp}=H^1_{\comp,\pl}\oplus H^1_{\comp,\CD}
\end{equation}
(the orthogonal decomposition in the metric  $\ba_\G$). We will
denote by $\varf_{\pl}$ and $\varf_\CD$ the components of a given element
$\varf$ with respect to this decomposition.
\subsection{Dimensions of the metric graph}
The following theorem  shows that the local dimension of the Laplacian on $\G$ is any number $\d>1$; it can be chosen  arbitrarily close to $1$ (that supports the intuitive understanding of the local dimension), while the  dimension at infinity is the same as it is for $G$. We  believe that the estimate (4.3) below is satisfied with
$\d=1$. However, a weaker result that
we prove in Theorem 4.1, $1^\circ$, is sufficient for our main
conclusions on the spectral estimates.
\begin{thm}\label{prop.dimension}
$1^\circ$ For any $\d>1$,
\begin{equation}\label{dimension0}
    M_{\BA_\G}(t)\le C(\d)t^{-\frac{\d}2},\ t\in(0,1).
\end{equation}
$2^\circ$ If the lengths of
the edges are uniformly bounded,
\begin{equation}\label{fromabove}
    l_e\le l_+,\ \forall e\in\CE,
\end{equation}
 and
$M_{\BA_G}(t)=O(t^{-\frac{D}2}),\ D>2$, as $t\to \infty$, then also $M_{\BA_{\G}}(t)=O(t^{-\frac{D}2})$.
\end{thm}
\begin{proof}
$1^\circ$ It suffices to prove \eqref{dimension0} for $\d\in(1,2)$. One can find $s>0$, such that for any point $z\in \G$ there exists a simple path $S(z)$ in $\G$ containing $z$ and having length $s$. To show this, fix a vertex $v_0\in\CV$ and set $s$ as the minimal length of edges containing $v_0$. For   $z\in\CS(v_0)$ the statement is obvious, for any other  $z\in\G$, take as $S(z)$ a segment with length $s$, containing $z$, of an arbitrary  simple path connecting $z$ with $v_0$.  Further on we treat such a path $S(z)$ as an interval.

For a fixed $z\in\G$,
consider the operator $T_z$ mapping a function $\varf$ on $\G$ to its restriction to $S(z)$. The operator $T_z$ is obviously bounded as acting from $H^1(\G)=\Dom(\BA_\G^{\frac12})$ to $H^1(S(z))$ and from $L^2(\G)$ to $L^2(S(z))$, with norms not greater than $1$. By interpolation,  $T_z$ is bounded as acting from $\Dom(\BA_\G^{\frac{\d}4})$ to $H^{\frac{\d}2}(S(z)),$ for any $\d\in(0,2)$, again with    norm    not greater  than    $1$.  For $\d>1$ the space $H^{\frac{\d}2}(S(z))$ is embedded in $C(S(z))$, with the same norm of the embedding operator for all $z$. We use the fact that $z$ is arbitrary   to conclude that the operator
$(\BI+\BA_\G)^{-\frac{\d}4}$ is bounded as acting from $L^2(\G)$ to $L^\infty(\G)$. Next, we have $$\exp(-t\BA_\G)=t^{-\frac{\d}4}(\BI+\BA_\G)^{-\frac{\d}4}\left[(t(\BI+\BA_\G))^{\frac{\d}4}\exp(-t\BA_\G)\right].$$
By the spectral theorem, the operator in brackets is bounded in $L^2(\G)$ uniformly in $t\in(0,1)$ and, therefore,  $$\|\exp(-t\BA_\G)\|_{L^2(\G)\to L^\infty(\G)}=O(t^{-\frac{\d}4}),\ t\in (0,1).$$
 This estimate, together with its dual,  imply  \eqref{dimension0}; see \cite{RS97}, section 2.1, for details and further references.\\
$2^\circ$ By Theorem II.3.1 in \cite{VCS}, it is sufficient to
prove that the Sobolev inequality
\begin{equation}\label{SobolevGamma}
    \|\varf\|^2_{L^p}\le C(p)\ba_\G[\varf],\qquad p=p(D)=2D/(D-2),
\end{equation}
holds for any $\varf\in H^1_{\comp}(\G)$.
Since the decomposition \eqref{1} is orthogonal, it is sufficient
to establish  \eqref{SobolevGamma}
 separately for the components $\varf_{\pl}$ and $\varf_{\CD}$.
 For the term  $\varf_{\pl}=Jf$, its norm in $L^p(\G)$ is
 majorized by the norm of $f$ in $\ell^p(G)$, so the
 Sobolev inequality for $\varf_{\pl}$ follows from the corresponding inequality
 for $f$. For  $\varf_\CD$, due to \eqref{fromabove}, the required  Sobolev
 inequality holds on each edge, with a common constant,  and the
 summation gives \eqref{SobolevGamma}.
\end{proof}
So, the general results of Section 2 apply to the metric graphs.
However, the analysis carried out below gives somewhat more
complete and detailed picture.

\subsection{Birman-Schwinger operators }
As always, we suppose that  the combinatorial graph $G$ satisfies \eqref{D}. Hence, by Theorem \ref{prop.dimension} and \eqref{embed},
 the
space $\CH^1=\CH^1(\G)$,  defined as the closure of
$H_{\comp}^1(\G)$ in the metric $\ba_\G$, is a Hilbert space of functions, embedded in $L^p(\G)$.

 Taking closure of both
terms in the decomposition \eqref{1} in the same metric, we obtain the Hilbert spaces
$\CH^1_{\pl}$ and $\CH^1_\CD$, so that
\[ {\CH^1_\CD(\G)}={\sum_{e\in\CE}}^\oplus H^{1,0}(e)\]
and
\begin{equation}\label{1a}
    \CH^1(\G)=\CH^1_{\pl}\oplus\CH^1_\CD.
\end{equation}
The isometry $J$ extends to the isometry of $\CH(G)$ onto
$\CH^1_{\pl}$.
The quadratic form \eqref{b} in our case is
\begin{equation}\label{2}
    \bb_V[\varf]=\int_{\G}V(y)|\varf(y)|^2dy.
\end{equation}
In general, the decomposition \eqref{1a} does not reduce the corresponding
operator $\BB_V$. Still, we introduce the operators $\BB_{V,\pl}$
and $\BB_{V,\CD}$, acting in $\CH^1_{\pl}$ and $\CH^1_\CD$ respectively
and generated by the \qf \eqref{2} restricted to the corresponding
subspace. The spectral estimates for $\BB_V$ easily reduce to the
ones for these two operators. Indeed, it is clear that $\BB_V$ is
bounded (compact) if and only if these two operators possess this
property. Moreover, due to the inequality
\[ \bb_V[\varf]\le 2(\bb_V[\varf_{\pl}]+\bb_V[\varf_\CD]),\]
we have (in the case of compactness)
\begin{gather}\label{2m}
    \max\{n(s,\BB_{V,\pl}),\ n(s,\BB_{V,\CD})\}\le n(s,\BB_V)\le\\\nonumber
    n({s}/2,\BB_{V,\pl})+n({s}/2,\BB_{V,\CD}).
\end{gather}
Now we are ready to proceed to the upper spectral estimates for the
operators $\BB_{V,D}$
and  $\BB_{V,\pl}$. They are given in Lemmas 4.2, 4.3. The resulting
estimates for our original operator $\BB_V$ will be formulated in the
next section 5. The lower estimates, showing that
the result is sharp, are also derived in section 5.

The structure of the operator $\BB_{V,\CD}$ is simple:
\begin{equation}\label{3a}
    \BB_{V,\CD}={\sum_{e\in\CE}}^\oplus\BB_{V,e,\CD}
\end{equation}
where $\BB_{V,e,\CD}$ stands for the operator in $H^{1,0}(e)$,
generated by the \qf similar to \eqref{2}, with the integration
over the edge $e$.

Consider now the \qf \eqref{2} for $\varf\in\CH^1_{\pl}(\G)$. Let
$f=\{f(v)\}$ be the restriction of $\varf$ onto $\CV$, i.e.,
$f(v)=\varf(v),\ \forall v\in\CV$. Then
\begin{equation}\label{2y}
    \bb_V[\varf_{\pl}]=\bb_V[Jf]=\sum_{e\in\CE}\int_e V(y)|(Jf)(y)|^2dy.
\end{equation}
The corresponding operator on $\CH^1_{\pl}$ is $\BB_{V,\pl}$.
Consider also the operator $\wh\BB_{V,\pl}$ in $\CH^1(G)$,
generated by the \qf $\wh{\bb}_V[f]\equiv\bb_V[Jf]$. It is clear that the operators
$\BB_{V,\pl}$ and $\wh\BB_{V,\pl}$ are unitarily equivalent.\vs

\subsection{Operator $\BB_{V,\CD}$}\label{Dir}
The orthogonal decomposition \eqref{3a} reduces the study of the
spectrum of $\BB_{V,\CD}$ to the same problem for a family of
finite intervals, and thus makes the task elementary. In what follows
we always assume that the condition \eqref{fromabove} is satisfied.\vs

We associate with $V$ the  sequence
\begin{equation}\label{2a}
    \bsymb\y_V=\{\y_V(e)\},\qquad \y_V(e)=l_e\int_e Vdy,\ e\in\CE.
\end{equation}
It is well known (see, e.g., the estimate (4.8) and Theorem 4.6 in \cite{BSlect},
where one has to take $l=m=1$), that
\begin{equation}\label{2k}
    n(\l,\BB_{V,e,\CD})\le C\l^{-1/2}\sqrt{\y_V(e)},\qquad \forall \l>0.
\end{equation}
and
\begin{equation}\label{2s} \l^{1/2}n(\l,\BB_{V,e,\CD})\to\frac1{\pi}\int_e\sqrt{V}dx,\qquad
\l\to 0.\end{equation}
Let $\nu(s,\bsymb\y_V)=\#\{e:\y_V(e)>s\},\ s>0$, be the
distribution function for the sequence \eqref{2a}. We say that
$\y_V\to0$ if $\nu(s,\bsymb\y_V)<\infty$ for any $s>0$.

The next statement follows from \eqref{3a}, due to \eqref{2k}, \eqref{2s}.
\begin{lem}\label{lD}
$1^\circ$ If $\bsymb\y_V\in\ell^\infty$ then
the operator $\BB_{V,\CD}$ is bounded and
\begin{equation}\label{2inf}
    \|\BB_{V,\CD}\|\le
C\|\bsymb\y_V\|_{\ell^\infty},\qquad C>0.
\end{equation}
If $\bsymb\y_V\to0$ then the operator $\BB_{V,\CD}$ is compact.

$2^\circ$ For any $q\in(\frac12,\infty)$,
\begin{equation}\label{2q}
    \|\BB_{V,\CD}\|_{\GS_q}\le
C(q)\|\bsymb\y_V\|_{\ell^q}, \ \|\BB_{V,\CD}\|_{\Sigma_q}\le
C(q)\|\bsymb\y_V\|_{\ell^q_w}.
\end{equation}

$3^\circ$ Let $\bsymb\y_V\in\ell^{1/2}$. Then
\begin{equation*}
    \|\BB_{V,\CD}\|_{\Sg_{1/2}}\le C\sum_{e\in\CE}\sqrt{\y_V(e)}
\end{equation*}
 and
\[ \l^{1/2}n(\l,\BB_{V,\CD})\to\frac1{\pi}\int_{\G}\sqrt{V}dy,\qquad
\l\to 0.\]
\end{lem}
\begin{proof}
The reasoning  is rather standard, see, e.g., \cite{NaiSol02}, and we
prove only the statement 2$^\circ$ for the classes $\Sigma_q$. If
$\bsymb\y_V\in\ell^{q}_w$, then, after an appropriate enumeration
of edges, $e_j$, we have :
\[ \y_V(e_j)\le Mj^{-1/q}.\]
 Hence, by \eqref{2k},
  \[n(\l,B_{V,e_j,\CD})\le
 CM^{1/2}\l^{-1/2}j^{-1/2q}.\]
 In particular, $n(\l,B_{V,e_j,\CD})=0$ if $j>C^{2q}M^q\l^{-q}$.
 Therefore,
 \[n(\l,B_{V,\CD})=\sum_e n(\l,B_{V,e,\CD})\le
 CM^{1/2}\l^{-1/2}\sum_{j\le C^{2q}M^q\l^{-q}}j^{-1/2q}\]
and, since $2q>1$,
\[n(\l,B_{V,\CD})\le C'M^q\l^{-q}, \]
whence the result.
\end{proof}

\subsection{Operator $\BB_{V,\pl}$}\label{pl}
We compare our operator $\BB_{V,\pl}$ (or, equivalently,
$\wh\BB_{V,\pl}$) with the operator $\BB_{\pmb{\vark}_V}$, where the
discrete potential $\pmb\vark_V=\{\vark_V(v)\}$ is chosen in a
special way:
\begin{equation}\label{4}
    \vark_V(v)=\int_{\CS(v)}Vdy=\sum_{e\ni v}l_e^{-1}\y_V(e),\qquad \forall v\in \CV.
\end{equation}

Let us return to the \qf in \eqref{2y}. Choose an edge $e=(v,v')$.
Identifying $e$ with the interval $(0,l_e)$, we have, for
$f\in\CH(G)$,
\begin{gather*}
\int_e V(y)|(Jf)(y)|^2dy=l_e^{-2}\int_0^{l_e}V(y)|f(v)(l_e-y)+f(v')y|^2dy\\
\le\max\{|f(v)|^2,|f(v')|^2\}\int_e V(y)dy.
\end{gather*}
Summing up the integrals over all $e\in\CE$, we see that
\begin{equation}\label{est1}
    \wh\bb_V[f]\le\bb_{\pmb\vark_V}[f].
\end{equation}

This leads us to the following result.
\begin{lem}\label{lpl} Suppose the operator $\BB_{\pmb{\vark}_V}$ on the combinatorial graph
$G(\G)$ is bounded, compact, or lies in one of the classes $\GS_q,\
\Sg_{q}$, $q>0$. Then the same is true for the operator $\BB_{V,\pl}$,
and the estimate
\begin{equation}\label{6}
    \|\BB_{V,\pl}\|\le \|\BB_{\pmb{\vark}_V}\|
\end{equation}
holds in the $($quasi$)$-norm of the corresponding class.
\end{lem}

\section{Metric graphs:  estimates in Schatten classes and asymptotics}\label{fin}
\subsection{Lower bounds}\label{low}
We start by a simple graph-theoretic lemma.
\begin{lem}\label{coloring} Let   condition \eqref{deg} be satisfied for  $G(\G)$.
Then the set $\CE$ can be split into
no more than $2\bd^2+1$ subsets $\CE_j$, such that
$\CS(e)\cap\CS(e')=\varnothing$ for any $e\ne e'$ in the
same $\CE_j$.
\end{lem}
\begin{proof}Order the edges in $\CE$ in an arbitrary way.
We must color the edges in $2\bd^2+1$ colors so that the stars of the edges of the same
color are disjoint. Suppose that we have already colored all edges
$e_k, k<n$. The star of  $e_n$ can have common edges with
no more than $2\bd^2$ stars of the previously colored edges. So we apply the unused color to $e_n$.
\end{proof}

Now we are in a position to give some lower estimates for our original
operator $\BB_V$. Similarly to \thmref{snizu},
they require  additional conditions: namely, we assume  that the edge lengths $l_e$
are bounded and separated from zero:
\begin{equation}\label{frombelow.1}
   0<l_-\le l_e\le l_+,\qquad \forall e\in\CE,
\end{equation}
and also that \eqref{deg} is satisfied. Note that the right inequality in \eqref{frombelow.1} coincides with \eqref{fromabove}, and that the
left inequality implies the condition \eqref{g_e}
for the combinatorial graph $G(\G)$. We also note that for the graphs satisfying
\eqref{frombelow.1} the (quasi-)norms of the sequences $\pmb\y_V$, $\pmb\vark_V$ in
any space $\ell^q$, $\ell^q_w$ are equivalent to each other.

\begin{lem}\label{lower} Suppose the conditions \eqref{deg} and \eqref{frombelow.1}
are satisfied. Then
 $1^\circ$ If the operator $\BB_V$ is bounded,
then $\bsymb\y_V\in\ell_\infty$ and
\[ \|\BB_V\|\ge c\|\bsymb\y_V\|_\infty.\]
$2^\circ$ If $\BB_V$ is compact, then the sequence $\bsymb\y_V$
tends to zero. Moreover, there exist constants $c',c''$ depending
on $\bd$ and on $l_+/l_-$, such that
\begin{equation}\label{2b}
    n(s,\BB_V)\ge c'\nu(c''s,\bsymb\y_V), \qquad\forall s>0.
\end{equation}
\end{lem}
\begin{proof}
The reasoning follows the scheme repeatedly used in the literature, see,
e.g., \cite{BS, NaiSol02, RS09}. Let $e=(v,v')$ be an edge.
Take a function $\varf_e\in\CH^1_{pl}(\G)$, such that $\varf_e(y)=1$ for
$y\in e$ and $\varf_e(y)=0$ outside the set
$\CS(e)=\CS(v)\cup\CS(v')$. Such function is unique, and
$\int_{\G}|\varf'_e|^2dy\le 2(\bd-1)l_-^{-1}$. Moreover,
$\int_{\G}V|\varf_e|^2dy\ge l_+^{-1}\y_V(e)$, and hence,
\begin{equation}\label{3c}
    \frac{\int_{\G}V|\varf_e|^2dy}{\int_{\G}|\varf_e'|^2dx}\ge (2\bd-2)^{-1}l_-l_+^{-1}\y_V(e).
\end{equation}
The statement $1^\circ$ immediately follows.

Further, if for two edges $e_1,e_2$ the sets $\CS(e_1),\CS(e_2)$
are disjoint, then the corresponding functions
$\varf_{e_1},\varf_{e_2}$ are orthogonal both in $\CH^1(\G)$ and in the
space $L^2$ with the weight $V$. Let us say that a subset
$F\subset\CE$ is nice, if the sets $\CS(e),\ e\in F$, are mutually
disjoint. By restricting the \qf $\bb_V$ onto the linear span of
the functions $\varf_e, e\in F$, we obtain an operator whose
eigenvalues, up to the ordering, are exactly the numbers in the
left-hand side of \eqref{3c}. Therefore, the numbers $\y_V(e)$,
listed in the decreasing order, do not exceed the eigenvalues
$\l_n(\BB_V)$ (or even, $\l_n(\BB_{V,\pl})$), multiplied by
$2(\bd-1)l_+l_-^{-1} $. By Lemma ~\ref{coloring}  the set
$\CE$ can be split into no more than $2\bd^2+1$  nice subsets,
 and this leads to the
statement $2^\circ$.
\end{proof}
 Note that a similar lower estimate for the operator
$\BB_{V,\CD}$ (in place of $\BB_V$) does not hold.

\subsection{The operator $\BB_V$}\label{whole}
Now we easily obtain the following result for our original operator $\BB_V$ in the space $\CH^1(\G)$,
generated by the quadratic form \eqref{2}. We compare it with the
`discrete' operator  $\BB_{\pmb{\vark}_V}$ in the space
$\CH(G),\ G=G(\G)$, where the discrete potential $\pmb{\vark}_V$ is defined by \eqref{4}.

\begin{thm}\label{wh}
Let $\G$ be a metric graph satisfying the conditions \eqref{deg} and
\eqref{frombelow.1}. Suppose also that $D>2$.
Then the operator $\BB_V$ is bounded (resp., compact,  lies in one of the
classes $\GS_q,\ \GS_{q,w}$ with $q>1/2$) if and only if the
following two conditions are satisfied.

$1^\circ$ The operator $\BB_{\pmb{\vark}_V}$ belongs to the corresponding class;

$2^\circ$ The sequence ${\pmb{\vark}_V}$ belongs to
$\ell^\infty$, resp., its subspace of sequences tending to zero,
$\ell^q$, or $\ell^q_w$.

If $q<D/2$, the condition $1^\circ$ follows from $2^\circ$ and thus, can be
removed.
\end{thm}

Note that the last statement follows from the upper estimates for the combinatorial graphs,
see \thmref{thm1}, $2^\circ$.

\begin{rem}In particular,   $\BB_V\in\GS_1$ if, and under assumptions \eqref{frombelow.1} and \eqref{deg}, only if
$\int_{\G}Vdy<\infty$.\end{rem}
\begin{proof}
Part `if' follows from Lemmas \ref{lD} and \ref{lpl}, due to the
second inequality in \eqref{2m}.

Part `only if': if $\BB_V$ possesses one of the properties mentioned
in the assumption, then the same is true for the operator
$\BB_{\pmb{\vark}_V}$ due to the first inequality in \eqref{2m}. The
condition 2$^\circ$ is fulfilled by Lemma \ref{lower}.
\end{proof}\vs

This result shows that the spectral properties of the operator $\BB_V$ are basically determined
by the ones for its discrete analogue. It is worth noting that the result for $q>D/2$ should be considered
as `conditional': indeed, our results for general combinatorial graphs (\thmref{thm1}) concern only the case $q\le D/2$. For more advanced results, one needs an additional information about the structure of $G$. For the important case $G=\Z^d,\ d\ge3$, such results were obtained in \cite{RS09}.

In particular,  a construction of `sparse potentials' was suggested there, that allows one to construct a discrete potential producing an operator $\BB_V$ with an arbitrary
prescribed asymptotic behavior of the spectrum.
In this connection, we note that this construction extends to arbitrary combinatorial
graphs with
$D>2$. This material will be presented elsewhere.

\thmref{wh} does not include the borderline case $q=1/2$. For this case, a simple sufficient
condition for
$\BB_V\in\GS_{1/2,w}$ can be given.
\begin{thm}\label{1/2}
Let $\bsymb\y_V\in\ell^{1/2}$. Then
\begin{equation}\label{estw}
    \|\BB_V\|_{\Sg_{1/2}}\le C\|\bsymb\y_V\|_{\ell^{1/2}},
\end{equation}
or, equivalently,
\[ N_-(\BA-\a V)\le C\a^{1/2}\sum_{e\in\CE}\bsymb\y_V(e)^{1/2}.\]
The Weyl-type asymptotic formula
\begin{equation}\label{weyl}
    N_-(\BA-\a V)\sim \frac{\a^{1/2}}{\pi}\int_\G\sqrt{V}dy
\end{equation}
is valid.
\end{thm}
\begin{proof}
By Lemma \ref{lD}, the estimate \eqref{estw} and the asymptotics \eqref{weyl} hold for the operator $\BB_{V,\CD}$. Since $\pmb{\vark}_V$ lies in $\ell^{1/2}$ together with $\bsymb{\y}_V$,
the operator $\BB_{V,\pl}$ belongs to $\GS_{1/2}$, and thus to $\Sg_{1/2}^{(0)}$,
with the corresponding estimate for its quasi-norm. Hence, it
it does not contribute to the asymptotics of order $1/2$. Both statements of Theorem follow
from this fact.
\end{proof}

A necessary and sufficient condition for the validity of \eqref{estw} and \eqref{weyl}
 can be obtained by analogy with
\cite{NaiSol97}.  We do not present it here.

\end{document}